\documentclass{fundam}
\usepackage{amssymb}
\usepackage{a4}
\usepackage{QED}
\begin{document}
\setcounter{page}{1001}
\issue{XXI~(2001)}


\newcommand {\es}{\emptyset}
\newcommand {\ra}{\tr}
\newcommand {\ras}{\tr^*}
\newcommand {\st}{ \rightarrow _{st}}
\newcommand {\g}{\gamma}
\newcommand {\G}{\Gamma}
\newcommand {\m}{\mu}
\newcommand {\de}{\delta}
\newcommand {\be}{\beta}
\newcommand {\al}{\alpha}
\newcommand {\sig}{\sigma}
\newcommand {\la}{\lambda}
\newcommand {\ptv}{\; ..., \;}
\newcommand {\pts}{ \; ... \; }
\newcommand {\peq}{\preceq}
\newcommand {\pe}{\prec}

\def\a{\alpha}
\def\b{\beta}
\def\g{\gamma}
\def\G{\Gamma}
\def\D{\triangle}
\def\t{\tau}
\def\d{\delta}
\def\th{\theta}
\def\l{\lambda}
\def\L{\Lambda}
\def\n{\nu}
\def\di{\diamond}
\def\r{\rho}
\def\ph{\phi}
\def\ps{\psi}
\def\ep{\epsilon}
\def\m{\mu}
\def\s{\sigma}
\def\sou{\overline}
\def\so{\underline}
\def\O{\Omega}
\def\o{\omega}
\def\ep{\varepsilon}
\def\f{\rightarrow}
\def\tr{\triangleright}
\def\q{\forall}
\def\e{\exists}
\def\v{\vdash}
\def\pe{\succeq}
\def\p{\succ}
\def\vi{\emptyset}
\def\ou{\vee}
\def\et{\wedge}
\def\<{\langle}
\def\>{\rangle}
\def\F{\displaystyle\frac}

\title{Arithmetical proofs of strong normalization results for symmetric $\lambda$-calculi}

\address{Ren\'e David.
Laboratoire de Mathé\'ematiques.
 Campus scientifique.
F-73376 Le Bourget du Lac. email : david@univ-savoie.fr}

\author{Ren\'e David \& Karim Nour\\
Laboratoire de Mathé\'ematiques\\
 Universit\'e de Savoie\\
73376 Le Bourget du Lac. France\\
\{david,nour\}@univ-savoie.fr} \maketitle

\runninghead{R. David, K. Nour}{Arithmetical proofs of strong normalization results}

\begin{abstract}
We give {\em arithmetical} proofs of the strong normalization of
two  symmetric $\lambda$-calculi corresponding to classical logic.\\
The first one is the $\overline{\lambda}\mu\tilde{\mu}$-calculus
introduced by Curien \& Herbelin. It is derived via the
Curry-Howard correspondence from Gentzen's classical sequent
calculus LK in order to have a symmetry on one side between
``program'' and ``context'' and on other side between
``call-by-name'' and ``call-by-value''. \\
The second one is the
symmetric $\lambda \mu$-calculus. It is the $\lambda \mu$-calculus
introduced by Parigot in which the reduction rule $\m'$, which is
the symmetric of $\m$, is added.\\
These results were already known but the previous proofs use
candidates of reducibility where the interpretation of a type is
defined as the fix point of some increasing operator and thus, are
highly non arithmetical.\\
\end{abstract}

\begin{keywords}
$\lambda$-calculus, symmetric calculi, classical logic, strong
normalization.
\end{keywords}

\section{Introduction}

Since it has been understood that  the Curry-Howard correspondence
relating proofs and programs  can be extended to classical logic
(Felleisen \cite{Fel}, Griffin \cite{Gri}), various  systems have
been introduced: the $\l_c$-calculus (Krivine \cite{Kri}), the
$\la_{exn}$-calculus (de Groote \cite{deG4}), the $\l
\mu$-calculus (Parigot \cite{Par1}), the $\lambda^{Sym}$-calculus
(Barbanera \& Berardi \cite{BaBe}), the
$\lambda_{\Delta}$-calculus (Rehof \& Sorensen \cite{ReSo}), the
$\overline{\lambda}\mu\tilde{\mu}$-calculus (Curien \& Herbelin
\cite{cuhe}), the dual calculus (Wadler \cite{wad1}), ... Only a
few of them have computation rules that correspond to the symmetry
of classical logic.

We consider here the $\overline{\lambda}\mu\tilde{\mu}$-calculus
and the symmetric $\l \mu$-calculus and we give arithmetical
proofs of the strong normalization of the simply typed calculi.
Though essentially the same proof can be done for the
$\lambda^{Sym}$-calculus, we do not consider here this calculus
since it is somehow different from the previous ones: its main
connector is not the arrow but the connectors {\em or} and {\em
and} and the symmetry of the calculus comes from the de Morgan
laws. This proof will appear in Battyanyi's PhD thesis \cite{Bat}
who will also consider the dual calculus. Note that  Dougherty \&
all \cite{Dou} have shown the strong normalization of this
calculus by the reducibility method using the technique of the
fixed point construction.

The first proof of strong normalization for a symmetric calculus
is the one by  Barbanera \& Berardi for the
$\lambda^{Sym}$-calculus. It uses candidates of reducibility but,
unlike the usual construction (for example for Girard's system
$F$), the definition of the interpretation of a type needs a
rather complex fix-point operation. Yamagata \cite{Yam} has used
the same technic to prove the strong normalization of the
symmetric $\la\m$-calculus where the types are those of system $F$
and Parigot, again using the same ideas, has extended Barbanera \&
Berardi's result to a logic with second order quantification.
Polonovsky, using the same technic, has proved in \cite{polo} the
strong normalization of the
$\overline{\lambda}\mu\tilde{\mu}$-reduction.  These proofs are
highly non arithmetical.

The two proofs that we give are essentially the same but the proof
for the $\overline{\lambda}\mu\tilde{\mu}$-calculus is much
simpler since some difficult problems that appear in the
$\l\m$-calculus do not appear in the
$\overline{\lambda}\mu\tilde{\mu}$-calculus. In the
$\overline{\lambda}\mu\tilde{\mu}$-calculus, a $\m$ or a $\l$
cannot be created at the root of a term by a reduction but this is
not the case for the symmetric $\l\m$-calculus. This is mainly due
to the fact that, in the former, there is a right-hand side and a
left-hand side whereas, in the latter, this distinction is
impossible since a term on the right of an application can go on
the left of an application after some reductions.

The idea of the proofs  given here comes from the one given by the
first author for the simply typed $\l$-calculus : assuming that a
typed term has an infinite reduction, we can define,  by looking
at some particular steps of this reduction,  an infinite sequence
of strictly decreasing types.  This proof can be found either in
\cite{dav1} (where it appears among many other things) or as a
simple unpublished note on the web page of the first author
 (\verb www.lama.univ-savoie.fr/~david ~).

We also show the strong normalization of the
$\mu\tilde{\mu}$-reduction (resp. the $\m\m'$-reduction) for the
un-typed calculi. The first result was already known and it can be
found in \cite{polo}. The proof is done (by using candidates of
reducibility and a fix point operator) for a typed calculus but,
in fact, since the type system is such that every term is typable,
the result is valid for every term. It was known that, for the
un-typed $\la\m$-calculus, the $\m$-reduction is strongly
normalizing (see \cite{Py}) but the strong normalization of  the
$\m\m'$-reduction was an open problem raised long ago by Parigot.
Studying this reduction by itself is interesting since a $\m$ (or
$\m'$)-reduction can be seen as  a way ``to put the arguments of
the $\m$ where they are used'' and it is useful to know that this
is terminating.

This paper is an extension of \cite{tlca}. In particular, section
4 essentially appears there. It is organized as follows. Section 2
gives the syntax of the terms of the
$\overline{\lambda}\mu\tilde{\mu}$-calculus and the symmetric
$\l\m$-calculus and their reduction rules. Section 3 is devoted to
the proof of the normalization results for the
$\overline{\lambda}\mu\tilde{\mu}$-calculus and section 4 for the
symmetric $\l\m$-calculus. We conclude in section 5 with some
remarks and  future work.

\section{The calculi}

\subsection{The
$\overline{\lambda}\mu\tilde{\mu}$-calculus}\label{1s2}

\subsubsection{The un-typed calculus}

There are three kinds of terms, defined by the following grammar,
and there are two kinds of variables.
 In the literature, different authors use different terminology. Here, we will call them either $c$-terms, or $l$-terms
or $r$-terms. Similarly, the variables will be called either
$l$-variables (and denoted as $x,y,...$) or  $r$-variables (and
denoted as $\al, \be, ...$). \\In the rest of the paper, by term
we will mean any of these three kind of terms.
\begin{center}

\begin{tabular}{ccccccccc}
$c$   &::= & $\langle t_l , t_r \rangle$ & \, &\, &\, &\, &\, &\, \\
$t_l$ &::= & $x$ &$\mid$ & $\l x \, t_l$ &$\mid$ &$\mu \a \, c$ &$\mid$ & $t_r.t_l$ \\
$t_r$ &::= &$\a$ &$\mid$ & $\l \a \, t_r$ &$\mid$ &$\mu  x \, c$
&$\mid$ & $t_l.t_r$
\end{tabular}
\end{center}

\begin{remark}
$t_l$ (resp. $t_r$) stands of course for the left (resp. right)
part of a $c$-term. At first look, it may be strange that, in the
typing rules below, left terms appear in the right part of a
sequent and vice-versa. This is just a matter of convention and an
other choice could have been done. Except the change of name (done
to  make easier the analogy between the proofs for
$\overline{\lambda}\mu\tilde{\mu}$-calculus  and the symmetric
$\l\m$-calculus) we have respected the notations of the literature
on this calculus.
\end{remark}

\subsubsection{The typed calculus}

The logical part of this calculus is the (classical) sequent
calculus which is, intrinsically, symmetric. The types are built
from atomic formulas  with the connectors $\rightarrow$ and $-$
where the intuitive meaning of $A-B$ is ``$A$ and not $B$''. The
typing system is a sequent calculus based on judgments of the
following form:
\begin{center}
$c : (\G \v \D)$   $\;\;\;\;\;\;\;\;\;$ $\G \v \fbox{$t_l : A$},
\D$ $\;\;\;\;\;\;\;\;\;$ $\G, \fbox{$t_r : A$} \v \D$
\end{center}
where $\G$ (resp. $\D$) is a $l$-context (resp. a $r$-context),
i.e. a set of declarations of the form $x : A$ (resp. $\al : A$)
where $x$ (resp. $\al$) is a $l$-variable (resp. a $r$-variable)
and $A$ is a
type.\\

\begin{minipage}[t]{200pt}
$\F{}{\G, x : A \v \fbox{$x : A$}\, , \D }$\\
\end{minipage}
\begin{minipage}[t]{200pt}
$\F{}{\G , \,\fbox{$\a : A$} \v \a : A , \D }$\\[0.5cm]
\end{minipage}

\begin{minipage}[t]{200pt}
$\F{\G, x : A \v \fbox{$t_l : B$}\, , \D}{\G\v \fbox{$\l x \, t_l : A \f B$}\,, \D}$\\
\end{minipage}
\begin{minipage}[t]{200pt}
$\F{\G \v \fbox{$t_l : A$}\, , \D \;\;\; \G , \,  \fbox{$t_r : B
$}\v
    \D}{\G , \, \fbox{$t_l.t_r : A \f B$} \v \D}$\\[0.5cm]
\end{minipage}

\begin{minipage}[t]{200pt}
$\F{\G \v \fbox{$t_l : A$} \, , \D \;\;\; \G ,\, \fbox{$t_r : B$}
\v \D}{\G \v \fbox{$ t_r.t_l : A - B$} \,, \D}$
\end{minipage}
\begin{minipage}[t]{200pt}
$\F{\G , \,  \fbox{$t_r : A$} \v \a : B , \D}{\G , \, \fbox{$\l \a \, t_r : A - B$} \v \D}$\\[0.5cm]
\end{minipage}

$$\F{\G \v \fbox{$t_l : A$}\, , \D \;\;\; \G \, , \fbox{$t_r : A$} \v
\D}{\langle t_l , t_r \rangle : (\G \v \D)}$$

\begin{minipage}[t]{220pt}
$\F{c : (\G \v \a : A , \D)}{\G \v \fbox{$\mu \a \, c : A$}\,,  \D}$\\
\end{minipage}
\begin{minipage}[t]{200pt}
$\F{c : (\G , x : A \v  \D)}{\G , \, \fbox{$\mu x \, c : A$} \v \D}$\\[0.5cm]
\end{minipage}

\subsubsection{The reduction rules}

The cut-elimination procedure (on the logical side) corresponds to
the reduction rules (on the terms) given below.

\begin{itemize}
\item  $\langle \l x \, t_l , t'_l.t_r \rangle \tr_{\;\l} \langle t'_l, \mu
  x \, \langle t_l , t_r \rangle \rangle$
\item  $\langle  t'_r.t_l , \l \a \, t_r \rangle \tr_{\;\overline{\l}}
  \langle \mu \a  \,  \langle t_l , t_r \rangle , t'_r\rangle$
\item $\langle \mu \a \, c , t_r \rangle \tr_{\;\mu} c[\a := t_r]$
\item  $\langle  t_l , \mu x \, c \rangle \tr_{\;\tilde{\mu}} c[x := t_l]$
\item  $ \mu \a \, \langle t_l , \a \rangle \tr_{\;s_l} t_l$
  $\;\;\;\;\;\;\;\;$ if $\a \not \in Fv(t_l)$
\item $\mu x \, \langle  x ,  t_r \rangle \tr_{\;s_r} t_r$   $\;\;\;\;\;\;\;$ if $x \not \in Fv(t_r)$
\end{itemize}

\begin{remark}

It is easy to show that the $\m\tilde{\mu}$-reduction is not
confluent. For example $\langle \m \al \, \langle x , \be \rangle
,$ $\m y \, \langle x , \al \rangle
\rangle$ reduces both to $\langle x , \be \rangle$ and to $\langle x , \al \rangle$.\\
\end{remark}

\begin{definition}
\begin{itemize}
  \item We denote by $\tr_{\;l}$ the reduction by one of the logical rules i.e.
$\tr_{\;\l}$, $ \tr_{\;\overline{\l}}$, $\tr_{\;\mu}$ or
$\tr_{\;\tilde{\mu}}$.
  \item We denote by $\tr_{\;s}$ the reduction by one of the
simplification rules  i.e. $\tr_{\;s_l}$ or $\tr_{\;s_r}$

\end{itemize}

\end{definition}

\subsection{The symmetric  $\l \mu$-calculus}

\subsubsection{The un-typed calculus}

The set (denoted as ${\cal T}$) of $\l\m$-terms or simply terms is
defined by the following grammar where $x,y,...$ are
$\lambda$-variables and $\al, \be, ...$ are $\mu$-variables:
$$
{\cal T} ::= x \mid \l x {\cal T} \mid ({\cal T} \; {\cal T}) \mid
\mu \al {\cal T} \mid   (\al \; {\cal T})
$$

Note that we adopt here a more liberal syntax (also called de
Groote's calculus) than in the original calculus since we do not
ask that a $\m \al$ is immediately followed by a $(\be \; M)$
(denoted $[\be] M$ in Parigot's notation).

\subsubsection{The typed calculus}

The logical part of this calculus is natural deduction.
 The types are those of the
simply typed $\l\m$-calculus i.e. are built from atomic formulas
and the constant symbol $\perp$ with the connector $\rightarrow$.
As usual $\neg A$ is an abbreviation for $A \rightarrow \perp$.

 The typing rules are given below
where $\G$ is a context, i.e. a set of declarations of the form $x
: A$ and $\al : \neg A$ where $x$ is a $\l$ (or intuitionistic)
variable, $\al$ is a $\m$ (or classical) variable and $A$ is a
formula.

\begin{center}
$\F{}{\G , x : A \v x : A} \, ax$

\medskip
$\F{\G, x : A \v M : B} {\G \v \l x M : A \f B} \, \f_i$
\hspace{0.5cm}  $\F{\G \v M : A \f B \quad \G \v N : A} {\G \v (M
\; N): B }\, \f_e$

\medskip

$\F{\G , \al : \neg A  \v M : \bot} {\G \v \mu \al M : A }  \,
\bot_e$ \hspace{0.5cm} $\F{\G , \al : \neg A  \v M : A} {\G, \al :
\neg A  \v ( \al \; M) : \bot } \, \bot_i$
\medskip

\end{center}

Note that, here, we also have changed Parigot's notation but these
typing rules are those of his classical natural deduction. Instead
of writing
$$M : (A_1^{x_1 }, ..., A_n^{x_n }  \vdash   B,  C_1^{\al_1},
..., C_m^{\al_m })$$ we have written
$$x_1 : A_1, ..., x_n : A_n,
\al_1: \neg C_1, ..., \al_m : \neg C_m \vdash M : B$$

\subsubsection{The reduction rules}

The cut-elimination procedure (on the logical side) corresponds to
the reduction rules (on the terms) given below. Natural deduction
is not, intrinsically, symmetric but Parigot has introduced the so
called {\em Free deduction} \cite{Pa01} which is completely
symmetric. The  $\l \mu$-calculus comes from there. To get a
confluent calculus he had, in his terminology, to fix the inputs
on the left. To keep the symmetry, it is enough  to add a new
reduction rule (called the $\mu'$-reduction) which is the
symmetric rule of the $\mu$-reduction and also corresponds  to the
elimination of a
 cut.

\begin{itemize}
\item
$(\l x M \; N) \tr_{\beta} M[x:=N]$

\item $(\mu \al M \; N) \tr_{\mu} \mu \al M[\al=_r N]$

\item  $(N \; \mu \al M) \tr_{\mu'} \mu \al M[\al=_l N]$

\item  $ (\al \; \m\be M) \tr_{\rho}
M[\be:=\al]$

\item  $ \m\al (\al \; M)
\tr_{\theta} M$  if $\al$ is not free in $M$.

\end{itemize}
where $M[\al=_r N]$ (resp. $M[\al=_l N]$) is obtained by replacing
each sub-term of $M$ of the form $(\al \; U)$ by $(\al \; (U \;
N))$ (resp. $(\al \; (N \; U))$). This substitution is called a
$\m$-substitution (resp. a $\m'$-substitution).

\begin{remark}

\begin{enumerate}
  \item It is shown in \cite{Par1} that the $\be\m$-reduction is confluent
but neither $\m\m'$ nor $\be\m'$ is. For example $(\m \al x \, \m
\be y)$ reduces both to $\m \al x$ and to $\m \be y$. Similarly
$(\la z  x \; \m \be y)$ reduces both to $x$ and to $\m \be y$.

\item Unlike for a $\be$-substitution where, in $M[x:=N]$, the variable $x$ has disappeared it
 is important to note that, in a $\m$ or $\m'$-substitution, the variable $\al$ has not disappeared. Moreover its type has
 changed.  If the type of $N$ is $A$ and, in $ M $, the type of $\al$
 is $\neg (A\rightarrow B)$ it becomes $\neg B$ in $M[\al=_r N]$. If
 the type of $N$ is $A \rightarrow B$ and, in $ M$, the type of $\al$
 is $\neg A$ it becomes $\neg B$ in $M[\al=_l N]$.

 \item In section 4, we will {\em not} consider the rules
 $\theta$ and $\rho$.
  The rule $\theta$ causes no problem
since it is strongly normalizing and it is easy to see that this
rule can be postponed. However, unlike for the
$\overline{\lambda}\mu\tilde{\mu}$-calculus where all the
simplification rules can be postponed, this is not true for the
rule $\rho$  and, actually, Battyanyi has shown in \cite{Bat} that
$\m\m'\rho$ is {\em not} strongly normalizing. However he  has
shown that $\m\m'\rho$ (in the untyped case) and $\beta\m\m'\rho$
(in the typed case) are {\em weakly} normalizing.

\end{enumerate}

\end{remark}

\subsection{Some notations}

The following notations will be used for both calculi. It will
also be important to note that, in section 3 and 4, we will use
the same notations (for example $\Sigma_l, \Sigma_r$) for objects
concerning respectively the
$\overline{\lambda}\mu\tilde{\mu}$-calculus and the symmetric
$\l\m$-calculus. This is done intentionally to show the analogy
between the proofs.

\begin{definition}
Let $u,v$ be terms.
\begin{enumerate}
\item $cxty(u)$ is the number of symbols occurring in $u$.
\item We denote by  $u \leq v$ (resp. $u < v$) the fact that $u$ is a sub-term
(resp. a strict sub-term) of $v$.
\item A {\em proper} term is a term that is not a variable.
\item If $\sig$ is a substitution and
$u$ is a term, we denote by

\begin{itemize}
  \item $\sig +[x:=u]$ the substitution
$\sig'$ such that for $y \neq x$, $\sig'(y)=\sig(y)$ and
$\sig'(x)=u$
  \item  $\sig[x:=u]$ the substitution $\sig'$ such
that $\sig'(y)=\sig(y)[x:=u]$.

\end{itemize}

\end{enumerate}
\end{definition}

\begin{definition}
Let $A$ be a type. We denote by $lg(A)$ the number of symbols in
$A$.
\end{definition}

In the next sections we will study various reductions. The
following notions will correspond to these reductions.

\begin{definition}
 Let $\tr$ be a notion of reduction.
\begin{enumerate}

\item  The
 transitive (resp. reflexive and transitive) closure of $\tr$ is
 denoted by $\tr^+$ (resp. $\tr^*$). The length (i.e. the number of
 steps) of the reduction $t \ras t'$ is denoted by $lg(t \ras t')$.

\item If $t$ is in $SN$
  i.e. $t$ has no infinite reduction, $\eta(t)$ will denote the length
of the longest reduction starting from $t$ and $\eta c(t)$ will
denote  $(\eta(t), cxty(t))$.
\item We denote by $u \prec v$ the fact that $u \leq w$ for some $w$ such that $v \ras w$
   and either $v \tr^+ w$ or $u < w$. We denote by $\preceq$ the
   reflexive closure of $\prec$.

\end{enumerate}
\end{definition}

\begin{remark}

- It is easy to check that the relation $\peq$ is transitive, that
$u \preceq v$ iff $u \leq w$ for some $w$ such that $v \ras w$. We
can also prove (but we will not use it) that the relation $\peq$
is an order on the set $SN$.

-  If $v \in SN$ and $u \prec v$, then $u \in SN$ and $\eta c(u) <
\eta c(v)$.

- In the proofs  done by induction on some $k$-uplet of integers,
the order we consider is the lexicographic order.
\end{remark}

\section{Normalization for the $\overline{\lambda}\mu\tilde{\mu}$-calculus}

The following lemma will be useful.
\begin{lemma}\label{1tri}
Let $t$ be a $l$-term (resp. a $r$-term). If $t \in SN$, then
$\langle t,\al \rangle \in SN$ (resp.  $\langle x , t\rangle \in
SN)$.
\end{lemma}

\begin{proof}
By induction on $\eta(t)$. Since $\langle t,\al\rangle \not\in
SN$, $\langle t,\al\rangle \tr u$ for some $u$ such that $u \not
\in SN$.  If $u = \langle t',\al \rangle$ where $t \tr t'$ we
conclude by the induction hypothesis since $\eta(t')< \eta(t)$. If
$t = \m \be \, c$ and $u=c[\be := \al] \not\in SN$, then $c \not
\in SN$ and $t \not \in SN$. Contradiction.
\end{proof}

\subsection{$\tr_{\;s}$ can be postponed}\label{1s}

\begin{definition}
\begin{enumerate}
\item Let $\tr_{\;\mu_0}, \tr_{\;\tilde{\mu}_0}$ be defined as follows:
\begin{itemize}
\item $\langle \mu \a \, c , t_r \rangle \tr_{\;\mu_0} c[\a := t_r]$
  $\;\;\;\;\;\;$ if $\a$ occurs at most once in $c$
\item  $\langle  t_l , \mu x \, c \rangle \tr_{\;\tilde{\mu}_0} c[x :=
  t_l]$
  $\;\;\;\;\;\;\;$ if $x$ occurs at most once in $c$
\end{itemize}
\item Let $\tr_{\;l_0} =  \tr_{\;\mu_0} \cup \tr_{\;\tilde{\mu}_0}$.
\end{enumerate}
\end{definition}

\begin{lemma}\label{1sl}
If $u \tr_{\;s} v \tr_{\;l} w$, then there is $t$ such that $u
\tr_{\;l} t \tr_{\;s}^* w$ or $u \tr_{\;l_0} t \tr_{\;l} w$.
\end{lemma}

\begin{proof}
By induction on $u$.
\end{proof}

\begin{lemma}\label{1slO}
If $u \tr_{\;s} v \tr_{\;l_0} w$, then  either $u \tr_{\;l_0} w$
or, for some $t$, $u
 \tr_{\;l_0} t \tr_{\;s} w$ or  $u \tr_{\;l_0} t \tr_{\;l_0} w$.
\end{lemma}

\begin{proof}
By induction on $u$.
\end{proof}

\begin{lemma}\label{1snlO}
If $u \tr_{\;s}^* v \tr_{\;l_0} w$ then, for some $t$,  $u
 \tr_{\;l_0}^+ t \tr_{\;s}^* w$ and $lg(u \tr_{\;s}^* v \tr_{\;l_0} w) \leq lg(u
 \tr_{\;l_0}^+ t \tr_{\;s}^* w)$.
\end{lemma}

\begin{proof}
By induction on $lg(u \tr_{\;s}^* v \tr_{\;l_0} w)$. Use lemma
\ref{1slO}.
\end{proof}

\begin{lemma}\label{1snl}
If $u \tr_{\;s}^* v \tr_{\;l} w$ then, for some  $t$,  $u
 \tr_{\;l}^+ t \tr_{\;s}^* w$ .
\end{lemma}

\begin{proof}
By induction on $lg(u \tr_{\;s}^* v \tr_{\;l} w)$. Use lemmas
\ref{1sl} and \ref{1snlO}.
\end{proof}

\begin{corollary}\label{1snlm}
$\tr_{\;s}$ can be postponed.
\end{corollary}

\begin{proof}
By  lemma \ref{1snl}.
\end{proof}

\begin{lemma}\label{1ssn}
The $s$-reduction is strongly normalizing.
\end{lemma}

\begin{proof}
If $u \tr_{\;s} v$, then $cxty(u) > cxty(v)$.
\end{proof}

\begin{theorem}\label{1post}
\begin{enumerate}
\item If $t$ is strongly normalizing for the $l$-reduction, then it is also  strongly normalizing
for the
  $ls$-reduction .
\item If $t$ is strongly normalizing for the $\mu\tilde{\mu}$-reduction, then it is also strongly normalizing for
the
  $\mu\tilde{\mu}s$-reduction.
\end{enumerate}
\end{theorem}

\begin{proof}
Use lemmas \ref{1ssn} and \ref{1snlm}. It is easy to check that
the lemma \ref{1snlm}  remains true if we consider only the
reduction rules $\mu$ and $\tilde{\mu}$.
\end{proof}

\subsection{The  $\mu\tilde{\mu}$-reduction is strongly
normalizing}\label{1mm'}

In this section we consider only the $\mu\tilde{\mu}$-reduction
and we restrict the set of terms to the following grammar.

\begin{center}
\begin{tabular}{cccccc}
$c$   &::= & $\langle t_l , t_r \rangle$ & \,  &\, \\
$t_l$ &::= & $x$  &$\mid$ &$\mu \a \, c$  \\
$t_r$ &::= & $\a$  &$\mid$ &$\mu  x \, c$
\end{tabular}
\end{center}

It is easy to check that, to prove the strong normalization of the
full calculus with the $\mu\tilde{\mu}$-reduction, it is enough to
prove the strong normalization of this restricted calculus.

Remember that we are, here,  in the un-typed caculus and thus our
proof does not use types but the strong normalization of this
 calculus actually follows from the result of the next section: it is easy to check that,
in this restricted calculus, every term is typable by any type, in
the context where the free variables are given this type. We have
kept this section since the main ideas of the proof of the general
case already appear here and this is done in a simpler situation.

The main point of the proof is the following.
 It is easy to show  that if $t
\in SN$ but $t[x:=t_l]\not\in SN$, there is some $\langle x, t_r
\rangle \prec t$ such that $t_r[x:=t_l] \in SN$ and  $\langle t_l,
t_r[x:=t_l]\rangle \not \in SN$.
 But this is not enough and we need a
stronger (and more difficult) version of this: lemma \ref{1crux}
ensures that, if $t[\sig] \in SN$ but $t[\sig][x:=t_l] \not\in SN$
then the real cause of non $SN$ is, in some  sense, $[x:=t_l]$.

 Having this result, we show, essentially by
induction on $\eta c(t_l) + \eta c(t_r)$, that if $t_l,t_r \in SN$
then $\langle t_l, t_r\rangle \in SN$. The point is that there is,
in fact, no deep interactions between $t_l$ and $t_r$ i.e. in a
reduct of $\langle t_l, t_r\rangle$ we always know what is coming
from $t_l$ and what is coming from $t_r$. The final result comes
then from a trivial induction on the terms.

\begin{definition}
\begin{itemize}
  \item We denote by $\Sigma_l$ (resp. $\Sigma_r$) the set of simultaneous substitutions of the form
$[x_1:=t_1,...,x_n:=t_n]$ (resp. $[\al_1:=t_1,...,\al_n:=t_n]$)
  where $t_1,...,t_n$ are proper $l$-terms ($r$-terms).
\item For $s \in \{l,r\}$, if
  $\sig=[\xi_1:=t_1,...,\xi_n:=t_n] \in \Sigma_s$, we denote by
  $dom(\sig)$ (resp. $Im(\sig)$) the set $\{\xi_1,...,\xi_n\}$
  (resp. $\{t_1,...,t_n\}$).
\end{itemize}
\end{definition}

\begin{lemma}\label{1l12}
Assume $t_l,t_r \in SN$ and $\langle t_l,t_r\rangle \not\in SN$.
Then either $t_l = \m \al \,c$ and $c[\al := t_r] \not\in SN$ or
$t_r = \mu x \,c$ and $c[x :=t_l] \not\in SN$.
\end{lemma}

\begin{proof}
By induction on $\eta(t_l)+\eta(t_r)$. Since $\langle
t_l,t_r\rangle \not\in SN$, $\langle t_l,t_r\rangle \tr t$ for
some $t$ such that $t \not \in SN$.  If $t = \langle
t'_l,t_r\rangle$ where $t_l \tr t'_l$, we conclude by the
induction hypothesis since $\eta(t'_l)+\eta(t_r) <
\eta(t_l)+\eta(t_r)$. If $t = \langle t_l,t'_r\rangle$ where $t_r
\tr t'_r$, the proof is similar. If $t_l = \m \al \, c$ and
$t=c[\al := t_r] \not\in SN$ or $t_r = \mu x \, c$ and $t=c[x
:=t_l] \not\in SN$, the result is trivial.
\end{proof}

\begin{lemma}\label{1crux}
\begin{enumerate}
\item Let $t$ be a term, $t_l$ a $l$-term and $\tau \in
\Sigma_l$. Assume $t_l \in SN$, $x$ is free in $t$ but not free in
$Im(\tau)$. If $t[\tau] \in SN$ but $t[\tau][x:=t_l] \not\in SN$,
there is $\langle x,t_r \rangle\prec t$ and $\tau' \in \Sigma_l$
such that $t_r[\tau'] \in SN$ and $\langle t_l , t_r[\tau']
\rangle \not\in SN$.
\item  Let $t$ be a term, $t_r$ a $r$-term and $\sig \in
\Sigma_r$. Assume $t_r \in SN$, $\al$ is free in $t$ but not free
in $Im(\sig)$. If $t[\sigma] \in SN$ but $t[\sigma][\al:=t_r]
\not\in SN$, there is $\langle t_l,\al \rangle\prec t$ and
$\sigma' \in \Sigma_r$ such that $t_l[\sigma'] \in SN$ and
$\langle t_l[\sigma'] , t_r\rangle \not\in SN$.
\end{enumerate}
\end{lemma}

\begin{proof} We prove the case (1) (the case (2) is similar). Note that $t_l$ is proper
 since $t[\tau] \in SN$,
$t[\tau][x:=t_l] \not\in SN$ and $x$ is not free in $Im(\tau)$.
Let $Im(\tau)= \{t_1,...,t_k\}$. Let ${\cal U}=\{u$ / $u$ is
proper and $u \peq t\}$ and ${\cal V}=\{v$ / $v$ is proper
 and $v \peq t_i$ for some $i\}$. Define inductively the
sets $\Sigma'_l$ and $\Sigma'_r$ of substitutions by the following
rules:

\medskip

\noindent $\rho \in \Sigma'_l$ iff $\rho = \es$ or $\rho=\rho' +
[y := v[\delta]]$ for some $l$-term $v \in {\cal V}$, $\delta  \in
\Sigma'_r$ and $\rho' \in \Sigma'_l$

 \noindent $\delta \in \Sigma'_r$ iff
$\delta = \es$ or $\delta=\delta' + [\beta := u[\rho]]$ for some
$r$-term $u \in {\cal U}$, $\rho \in \Sigma'_l$ and $\delta' \in
\Sigma'_r$

\medskip

\noindent Denote by C the conclusion of the lemma, i.e. there is
$\langle x,t_r \rangle\prec t$ and $\tau' \in \Sigma_l$ such that
$t_r[\tau'] \in SN$ and $\langle t_l , t_r[\tau'] \rangle \not\in
SN$.
 We prove something more general.

\medskip

\noindent (1) If $u \in {\cal U}$, $\rho \in \Sigma'_l$,
$u[\rho]\in SN$ and $u[\rho][x:=t_l]\not\in SN$, then C holds.

\noindent (2) If $v \in {\cal V}$, $\delta \in \Sigma'_r$,
$v[\delta]\in SN$ and $v[\delta][x:=t_l]\not\in SN$, then C holds.
\medskip

The term $t$ is proper since $t[\tau][x:=t_l] \not \in SN$. Then
conclusion C follows from (1) with $t$ and $\tau$.

 The
properties (1) and (2) are proved by a simultaneous induction on
$\eta c(u[\rho])$ (for the first case) and $\eta c(v[\delta])$
(for the second case).
 We only consider  (1), the case (2) is proved in a similar way.
\begin{itemize}
  \item
 If $u$ begins with a $\m$. The result follows from the induction
hypothesis.
  \item If $u = \langle u_l,u_r \rangle$.

\begin{itemize}
  \item If $u_r[\rho][x:=t_l] \not \in
SN$:  then $u_r$ is proper and the result follows from the
induction hypothesis.
  \item If
$u_l[\rho][x:=t_l] \not \in SN$ and $u_l$ is proper: the result
follows from the induction hypothesis.
  \item If $u_l[\rho][x:=t_l]
\not \in SN$ and  $u_l = y \in dom(\rho)$. Let $\rho(y) = \mu \be
\, d [\delta]$, then $\mu \be d[\delta][x:=t_l] \not \in SN$ and
the result follows from the induction hypothesis with $\mu \be d$
and $\delta$ (case (2)) since $\eta c(\mu \be d[\delta]) < \eta
c(u[\rho])$.
\item Otherwise, by lemma \ref{1l12}, there are two cases to consider. Note that $u_r$ cannot be a variable because, otherwise, $u[\rho][x:=t_l] = \langle
u_l[\rho][x:=t_l], u_r \rangle$ and thus, by lemma \ref{1tri},
$u[\rho][x:=t_l]$ would be in $SN$.

\hspace{.5cm} (1) $u_l[\rho][x:=t_l] = \mu \al \, c$ and $c[\al :=
u_r[\rho][x:=t_l]] \not \in SN$.

 \hspace{1.5cm} -  If $u_l = \mu \al \, d$, then $d[\al := u_r][\rho][x:=t_l] \not
\in SN$ and  the result follows from the induction hypothesis with
$ d[\al := u_r]$ and
 $\rho$ since $\eta( d[\al := u_r][\rho]) < \eta(u[\rho])$.

 \hspace{1.5cm} - If $u_l = y
 \in dom(\rho)$, let $\rho(y) = \mu \be \, d [\delta]$, then
 $d[\delta'][x:=t_l] \not \in SN$ where $\delta' =
\delta + [\beta :=  u_r[\rho]]$ and the result follows
 from the induction hypothesis with $d$ and $\delta'$ (case(2)).

\hspace{1.5cm} - If $u_l = x$, then $\langle x,u_r \rangle$ and
$\tau'=\rho[x:=t_l]$ satisfy the desired conclusion.

\hspace{.5cm} (2) $u_r[\rho][x:=t_l] = \mu y \, c$ and $c[\al :=
u_l[\rho][x:=t_l]] \not \in SN$. Then  $u_r = \mu y \, d$ and $d[y
:= u_l][\rho][x:=t_l] \not \in SN$. The result follows from the
induction hypothesis with $d[y := u_l]$ and
 $\rho$ since $\eta(d[y := u_l][\rho]) < \eta(u[\rho])$.\qed
 \end{itemize}
 \end{itemize}
\end{proof}

\begin{theorem}\label{1thm'1}
The  $\mu\tilde{\mu}$-reduction is strongly normalizing.
\end{theorem}

\begin{proof}
By induction on the term. It is enough to show that, if $t_l,t_r
\in SN$, then $\langle t_l , t_r \rangle \in SN$. We prove
something more general: let $\sigma$ (resp. $\tau$) be in
$\Sigma_r$ (resp. $\Sigma_l$) and assume $t_l[\sigma],t_r[\tau]
\in SN$. Then $\langle t_l[\sigma],t_r[\tau]\rangle \in SN$.
Assume it is not the case and choose some elements such that
$t_l[\sigma],t_r[\tau] \in SN$, $\langle
t_l[\sigma],t_r[\tau]\rangle \not\in SN$ and
$(\eta(t_l)+\eta(t_r), cxty(t_l)+cxty(t_r))$ is minimal. By lemma
\ref{1l12}, either $t_l[\sigma] = \mu \al \, c$ and $c[\al :=
t_r[\tau]] \not \in SN$ or $t_r[\tau] = \mu x \, c$ and $c[x :=
t_l[\sigma]] \not \in SN$. Look at the second case (the first one
is similar). We have $t_r = \mu x \, d$ and $d[\tau] = c$, then
$d[\tau][x:=t_l[\sigma]] \not \in SN$. By lemma \ref{1crux}, let
$u_r\prec d$ and $\tau' \in \Sigma_l$ be such that $u_r[\tau'] \in
SN$, $\langle t_l[\sigma] , u_r[\tau'] \not\in SN$.  This
contradicts the minimality of the chosen elements since $\eta
c(u_r) < \eta c(t_r)$.
\end{proof}

\subsection{The typed $\overline{\lambda}\mu\tilde{\mu}$-calculus is strongly normalizing}\label{1beta}

In this section, we consider the typed calculus with the
$l$-reduction. By theorem \ref{1post}, this is enough to prove the
strong normalization of the full calculus. To simplify  notations,
we do not write explicitly the type information but, when needed,
we denote by $type(t)$ the type of the term $t$.

The proof is essentially the same as the one of theorem
\ref{1thm'1}. It relies on  lemma \ref{1crux'} for which type
considerations are needed:
 in its proof, some cases cannot be proved ``by themselves''
and we need an argument using the types. For this reason, its
proof is done using the additional fact that we already know that,
if $t_l,t_r \in SN$ and the type of $t_r$ is small, then
$t[x:=t_r]$ also is in $SN$. Since the proof of lemma \ref{1thm}
is done by induction on the type, when we will use lemma
\ref{1crux'}, the additional hypothesis will be available.
\begin{lemma}\label{1l30}
Assume $t_l,t_r \in SN$ and $\langle t_l,t_r\rangle \not\in SN$.
Then either ($t_l = \m \al \,c$ and $c[\al := t_r] \not\in SN$) or
($t_r = \mu x \,c$ and $c[x :=t_l] \not\in SN$) or ($t_l = \l x
u_l$, $t_r = u'_l.u_r$ and $\langle u'_l,\mu x \langle u_l , u_r
\rangle \rangle \not\in SN$) or  ($t_r = \l \al u_r$, $t_l =
u'_r.u_l$ and $\langle \mu \al \langle u_r , u_l \rangle,u'_r
\rangle \not\in SN$).
\end{lemma}

\begin{proof}
By induction on $\eta(t_l)+\eta(t_r)$.
\end{proof}

\begin{definition}
Let $A$ be a type. We denote  $\Sigma_{A,l}$ (resp.
$\Sigma_{A,r}$) the set of substitutions of the form
$[x_1:=t_1,...,x_n:=t_n]$ (resp. $[\al_1:=t_1,...,\al_n:=t_n]$)
  where $t_1,...,t_n$ are proper $l$-terms (resp. $r$-terms)  and the type of the $x_i$ (resp. $\al_i$) is $A$.

\end{definition}

\begin{lemma}\label{1crux'}
Let $n$ be an integer and $A$ be a type such that $lg(A)=n$.
Assume $H$ holds where $H$ is:
 for every $u,v \in SN$ such that $lg(type(v))<n$, $u[x:=v]
\in SN$.
\begin{enumerate}
\item Let $t$ be a term, $t_l$ a $l$-term and $\tau \in
\Sigma_{A,l}$. Assume $t_l \in SN$ and has type $A$, $x$ is free
in $t$ but not free in $Im(\tau)$. If $t[\tau] \in SN$ but
$t[\tau][x:=t_l] \not\in SN$, there is $\langle x,t_r \rangle\prec
t$ and $\tau' \in \Sigma_{A,l}$ such that $t_r[\tau'] \in SN$ and
$\langle t_l , t_r[\tau'] \rangle \not\in SN$.
\item  Let $t$ be a term, $t_r$ a $r$-term and $\sig \in
\Sigma_{A,r}$. Assume $t_r \in SN$ and has type $A$, $\al$ is free
in $t$ but not free in $Im(\sig)$. If $t[\sigma] \in SN$ but
$t[\sigma][\al:=t_r] \not\in SN$, there is $\langle t_l,\al
\rangle\prec t$ and $\sigma' \in \Sigma_{A,r}$ such that
$t_l[\sigma'] \in SN$ and $\langle t_l[\sigma'] , t_r\rangle
\not\in SN$.
\end{enumerate}
\end{lemma}

\begin{proof}
 We only prove the case (1),
the other one is similar. Note that $t_l$ is proper since $t[\tau]
\in SN$ and $t[\tau][x:=t_l] \not\in SN$.  Let $Im(\tau)=
\{t_1,...,t_k\}$. Let ${\cal U}=\{u$ / $u$ is proper and $u \peq
t\}$ and ${\cal V}=\{v$ / $v$ is proper and $v \peq t_i$ for some
$i\}$. Define inductively the sets $\Sigma'_{A,l}$ and
$\Sigma'_{A,r}$ of substitutions by the following rules:

\medskip

\noindent $\rho \in \Sigma'_{A,l}$ iff $\rho = \es$ or $\rho=\rho'
+ [y := v[\delta]]$ for some $l$-term $v \in {\cal V}$, $\delta
\in \Sigma'_{A,r}$, $\rho' \in \Sigma'_{A,l}$ and $y$ has type
$A$.

\noindent $\delta \in \Sigma'_{A,r}$ iff $\delta = \es$ or
$\delta=\delta' + [\beta := u[\rho]]$ for some $r$-term $u \in
{\cal U}$, $\rho \in \Sigma'_{A,l}$, $\delta' \in \Sigma'_{A,r}$
and $\beta$ has type $A$.

\medskip

\noindent Denote by C the conclusion of the lemma, i.e. there is
$\langle x,t_r \rangle\prec t$ and $\tau' \in \Sigma_{A,l}$ such
that $t_r[\tau'] \in SN$ and $\langle t_l , t_r[\tau'] \rangle
\not\in SN$.
 We prove something more general.

 \medskip

\noindent (1) If $u \in {\cal U}$, $\rho \in \Sigma'_{A,l}$,
$u[\rho]\in SN$ and $u[\rho][x:=t_l]\not\in SN$, then C holds.

\noindent (2) If $v \in {\cal V}$, $\delta \in \Sigma'_{A,r}$,
$v[\delta]\in SN$ and $v[\delta][x:=t_l]\not\in SN$, then C holds.

\medskip

Note that, since $t[\tau][x:=t_l] \not \in SN$,  $t$ is proper and
thus, C  follows from (1) with $t$ and $\tau$. The properties (1)
and (2) are proved by a simultaneous induction on $\eta
c(u[\rho])$ (for the first case) and $\eta c(v[\delta])$ (for the
second case). We only consider (1) since (2) is similar.

The proof is as in lemma \ref{1crux}. We only consider the
additional cases:
 $u = \langle u_l,u_r \rangle$, $u_l[\rho][x:=t_l]\in
SN$, $u_r[\rho][x:=t_l]\in SN$, $u_r$ is proper and one of the two
following cases occurs.

\begin{itemize}
\item $u_l[\rho][x:=t_l] = \l x v_l$, $u_r[\rho][x:=t_l] = v'_l.v_r$
and $\langle v'_l,\m x \langle v_l , v_r \rangle \rangle \not \in
SN$. Then,   $u_r = w'_l.w_r$, $v'_l = w'_l[\rho][x:=t_l] $ and
$v_r = w_r[\rho][x:=t_l]$. There are three cases to consider.
\begin{itemize}
  \item $u_l = \l x w_l$ and $w_l[\rho][x:=t_l] = v_l$, then the result
follows from the induction hypothesis with $\langle w'_l,\m x
\langle w_l , w_r \rangle \rangle$ and $\rho$ since $\eta(\langle
w'_l,\m x \langle w_l , w_r \rangle \rangle [\rho]) <
\eta(u[\rho])$.

  \item $u_l = y \in dom(\rho)$. Let $\rho(y) = \l z w_l[\d]$, then
$a= \langle w'_l[\rho],\m x \langle w_l[\d] , w_r[\rho] \rangle
\rangle$  $ [x:=t_l] \not \in SN$. But,

\hspace{1cm} - $b=w'_l[\rho][x:=t_l],
c=w_l[\d][x:=t_l],d=w_r[\rho][x:=t_l] \in SN$,

\hspace{1cm} - $lg(type(b)) <n$, $lg(type(c)) <n$,

\hspace{1cm} - $ a = \langle x_2,\m x \langle x_1, d \rangle
\rangle [x_1:=c][x_2:= b]$

and this contradicts
 the hypothesis $(H)$.
  \item $u_l = x$, then $\langle x , u_r \rangle$ and $\t' = \t [x:=t_l]$
   satisfy the desired conclusion.
\end{itemize}

\item $u_l[\rho][x:=t_l] = v'_r.v_l$, $u_r[\rho][x:=t_l] = \l \al v_r$
and $\langle \m \al \langle v_l , v_r \rangle,  v'_r \rangle \not
\in SN$. The proof is similar.\qed
\end{itemize}
\end{proof}

\begin{lemma}\label{1thm}
If $t,t_l,t_r \in SN$, then $t[x:=t_l],t[\al:=t_r]\in SN$.
\end{lemma}

\begin{proof}
We prove something a bit more general: let $A$ be a type and $t$ a
term.

\noindent (1) Let $t_1, ..., t_k$ be $l$-terms and $\tau_1, ...,
\tau_k$ be substitutions in $\Sigma_{A,r}$. If, for each $i$,
$t_i$ has type $A$ and $t_i[\tau_i] \in SN$, then
$t[x_1:=t_1[\tau_1], \ptv x_k:=t_k[\tau_k]] \in SN$.

\noindent (2) Let $t_1, ..., t_k$ be  $r$-terms and $\tau_1, ...,
\tau_k$ be substitutions in $\Sigma_{A,l}$. If, for each $i$,
$t_i$ has type $A$ and $t_i[\tau_i] \in SN$, then
$t[\al_1:=t_1[\tau_1], \ptv \al_k:=t_k[\tau_k]] \in SN$.

We only consider (1) since (2) is similar. This is proved by
induction on $(lg(A),$  $ \eta(t), cxty(t)$, $\Sigma \; \eta(t_i),
\Sigma \; cxty(t_i))$ where, in $\Sigma \; \eta(t_i)$ and $\Sigma
\; cxty(t_i)$, we count each occurrence of the substituted
variable. For example if $k=1$ and $x_1$ has $n$ occurrences,
$\Sigma \; \eta(t_i)=n.\eta(t_1)$.

The only no trivial case is $t = \langle u_l,u_r \rangle$. Let
$\sigma = [x_1:=t_1[\tau_1], \ptv x_k:=t_k[\tau_k]]$. By the
induction hypothesis, $u_l[\sigma], u_r[\sigma] \in SN$. By lemma
\ref{1l30}, there are four cases to consider.

\begin{itemize}
\item $u_l[\sigma] = \m \al c$ and $c[\al := u_r[\sigma]] \not \in SN$.

\begin{itemize}
  \item If $u_l = \m \al d$ and $d[\sigma] = c$. Then $d[\al := u_r][\sigma] \not \in SN$
   and, since $\eta(d[\al := u_r]) < \eta(t)$, this contradicts the
   induction hypothesis.

  \item If $u_l = x_i$, $t_i = \m \al d$ and $d[\t_i][\al := u_r[\sigma]] \not
   \in SN$. By lemma \ref{1crux'}, there is $v_l \peq d$ and $\t'_i \in
   \Sigma_{A,r}$ such that $v_l[\t'_i] \in SN$ and $\langle v_l[\t'_i],
   u_r[\sigma] \rangle \not \in SN$. Let $t' = \langle y , u_r \rangle$ where
   $y$ is a fresh variable and $\sigma' = \sigma + [y =
   v_l[\t'_i]]$. Then  $\langle v_l[\t'_i],
   u_r[\sigma] \rangle = t'[\sigma']$ and, since
   $(\eta(v_l),cxty(v_l)) < (\eta(t_i),cxty(t_i))$ we get a contradiction
   from the induction hypothesis.

\end{itemize}

\item $u_r[\sigma] = \m x c$ and $c[x :=u_l[\sigma]] \not \in SN$,
  then $u_r = \m x d$, $d[\sigma] = c$ and $d[x := u_l][\sigma] \not
  \in SN$. Since $\eta(d[x := u_l]) < \eta(t)$, this contradicts the
  induction hypothesis.
 \item $u_l[\sigma] = \l x v_l$, $u_r[\sigma] = v'_l.v_r$ and $\langle
  v'_l, \mu x \langle v_l , v_r \rangle \rangle \not \in SN$, then
  $u_r = w'_l.w_r$, $w'_l[\sigma] = v'_l$ and $w_r[\sigma] = v_r$.

\begin{itemize}
  \item If $u_l = \l x w_l$ and $w_l[\sigma]=v_l$. Then $\langle w'_l, \mu x
  \langle w_l , w_r \rangle \rangle [\sigma] \not \in SN$ and this contradicts the induction hypothesis, since $\eta(\langle w'_l, \mu x \langle w_l , w_r \rangle \rangle) <
  \eta(t)$.
  \item If $u_l = x_i$, $t_i = \l x w_l$ and $\langle w'_l [\sigma], \mu x
  \langle w_l[\t_i] , $  $w_r [\sigma] \rangle \rangle \not \in SN$. Then, $\langle w_l[\t_i], w_r[\sigma]
  \rangle = \langle y , u_r[\sigma] \rangle [y:=w_l[\t_i]]$
  where
  $y$ is a fresh variable and thus $\langle w_l[\t_i] , w_r
  [\sigma] \rangle \in SN$, since
  $lg(type( w_l[\t_i])) < lg(A)$.\\ Since $\langle w'_l [\sigma], \mu x \langle w_l[\t_i] , w_r [\sigma]
  \rangle \rangle = \langle z, \mu x \langle w_l[\t_i] , w_r [\sigma]
  \rangle \rangle [z := w'_l [\sigma]]$ where $z$ is a fresh variable and
$lg(type( w'_l [\sigma]))
  < lg(A)$, this contradicts the induction hypothesis.

\end{itemize}

\item $u_r[\sigma] = \l \al v_r$, $u_l[\sigma] = v'_r.v_l$ and
  $\langle \mu \al \langle v_l, v_r\rangle , v'_r\rangle \not \in
  SN$.  This  is proved in the same way.\qed
\end{itemize}
\end{proof}

\begin{theorem}\label{1thm1}
Every typed term is in $SN$.
\end{theorem}

\begin{proof}
 By induction on the term. It is enough to show that if $t_l,t_r \in
 SN$, then $\langle t_l , t_r \rangle \in SN$. Since $\langle t_l ,
 t_r \rangle= \langle x , \al \rangle[x:=t_l][\al:=t_r]$ where $x,\al$
 are fresh variables, the result follows from lemma
\ref{1thm}.
\end{proof}

\section{Normalization for the symmetric $\l\m$-calculus }\label{s2}

\subsection{The  $\mu\mu'$-reduction is strongly normalizing}\label{mm'}

In this section we consider the $\m\m'$-reduction, i.e. $M \tr M'$
means $M'$ is obtained from $M$ by one step of the
$\m\m'$-reduction. The proof of theorem \ref{thm1} is essentially
the same as the one of theorem \ref{1thm'1}. We first show (cf.
lemma \ref{l7}) that a $\m$ or $\m'$-substitution cannot create a
$\m$ and then we show (cf. lemma \ref{crux})  that, if $M[\sig]
\in SN$ but $M[\sig][\al=_rP] \not\in SN$, then the real cause of
non $SN$ is, in some  sense, $[\al=_rP]$. The main point is again
that, in a reduction of $(M \; N) \in SN$, there is, in fact, no
deep interactions between $M$ and $N$ i.e. in a reduct of $(M \;
N)$ we always know what is coming from $M$ and what is coming from
$N$.

\begin{definition}\rm
\begin{itemize}
  \item The set of simultaneous substitutions of
  the form $[\al_1=_{s_1}P_1  \ptv$  $ \al_n=_{s_n}P_n]$ where $s_i \in \{l,r\}$ will be denoted by
  $\Sigma$.
  \item  For  $s \in \{l,r\}$, the set of simultaneous substitutions of
  the form $[\al_1=_sP_1 $ ...$\al_n=_sP_n]$  will be denoted by
  $\Sigma_s$.
  \item If $\sig=[\al_1=_{s_1}P_1  \ptv$  $ \al_n=_{s_n}P_n]$, we denote by $dom(\sig)$
  (resp. $Im(\sig)$) the set $\{\al_1, \ptv \al_n\}$
  (resp. $\{P_1, \ptv P_n\}$ ).
\item Let $\sig \in \Sigma$. We say that $\sigma \in SN$ iff for every
$N \in Im(\sig)$, $N \in SN$.
\item If $\overrightarrow{P}$ is a sequence $P_1,...,P_n$ of terms,  $(M \; \overrightarrow{P})$
  will denote $(M \; P_1 \pts P_n)$.
\end{itemize}


\end{definition}

\begin{lemma}\label{ll2}
If $(M \; N) \ras \m \al P$, then either $M \ras \m \al M_1$ and
$M_1[\al =_r N] \ras P$ or $N \ras \m \al N_1$ and $N_1[\al =_l M]
\ras P$.
\end{lemma}

\begin{proof}
By induction  on the length of the reduction $(M \; N) \ras \m \al
P$.
\end{proof}

\begin{lemma}\label{l7}
 Let $M$ be a term and $\sigma \in \Sigma$. If
 $M[\sigma] \ras \m\al P$, then $M\ras \m\al Q$ for some $Q$ such that
 $Q[\sigma] \ras P$.
\end{lemma}

\begin{proof}
By induction on $M$. $M$ cannot be of the form $(\beta \, M')$ or
$\l x \, M'$. If $M$ begins with a $\m$, the result is trivial.
Otherwise $M=(M_1 \; M_2)$ and, by lemma \ref{ll2}, either
$M_1[\sigma] \ras \m\al R$ and $R[\al=_rM_2[\sigma]] \ras P$ or
$M_2[\sigma] \ras \m\al R$ and $R[\al=_lM_1[\sigma]] \ras P$. Look
at the first case (the other one is similar). By the induction
hypothesis $M_1\ras \m\al Q$ for some $Q$ such that $Q[\sigma]
\ras R$ and thus $M \ras \m\al Q[\al=_r M_2]$. Since $Q[\al=_r
M_2][\sigma] = Q[\sigma][\al=_r M_2[\sigma]] \ras R[\al=_r
M_2[\sigma]]\ras P$ we  are done.
\end{proof}

\begin{lemma}\label{l12}
Assume $M,N \in SN$ and $(M \; N) \not\in SN$. Then either $M \ras
\m \al M_1$ and $M_1[\al =_r N] \not\in SN$ or $N \ras \m \be N_1$
and $N_1[\be =_l M] \not\in SN$.
\end{lemma}

\begin{proof}
By induction on $\eta(M)+\eta(N)$. Since $(M \; N) \not\in SN$,
$(M \; N) \tr P$ for some $P$ such that  $P \not \in SN$.  If $P =
(M' \; N)$ where $M \tr M'$ we conclude by the induction
hypothesis since $\eta(M')+\eta(N) < \eta(M)+\eta(N)$. If $P = (M
\; N')$ where $N \tr N'$ the proof is similar. If $M = \mu \al
M_1$  and $P = \mu \al M_1[\al=_rN]$ or  $N = \mu \be N_1$ and $P
= \mu \be N_1[\be=_lM]$
 the result is trivial.
\end{proof}

\begin{lemma}\label{crux}

Let $M$ be a term and $\sig \in \Sigma_s$. Assume $\delta$ is free in
$M$ but not free in $Im(\sig)$. If $M[\sigma] \in SN$ but
$M[\sigma][\delta=_sP] \not\in SN$, there is $M'\prec M$ and $\sigma'$
such that $M'[\sigma'] \in SN$ and, if $s=r$, $(M'[\sigma'] \;\; P)
\not\in SN$ and, if $s=l$, $(P \;\; M'[\sigma']) \not\in SN$.
\end{lemma}

\begin{proof}
Assume $s=r$ (the other case is similar). Let  $Im(\sigma)= \{N_1,
\ptv N_k\} $. Assume $M,\delta, \sigma, P$ satisfy the hypothesis.
Let ${\cal U}=\{U \; / \; U \peq M\}$ and ${\cal V}=\{V \; / \; V
\peq N_i$ for some $i\}$. Define inductively the sets $\Sigma_m$
and $\Sigma_n$ of substitutions by the following rules:

$\rho \in \Sigma_m$ iff $ \rho = \es$ or $\rho=\rho' + [\be =_r
V[\tau]]$ for some $V \in {\cal V} $, $\tau  \in \Sigma_n$ and
$\rho' \in \Sigma_m$

$\tau \in \Sigma_n$ iff $ \tau = \es$ or $\tau=\tau' + [\al =_l
U[\rho]]$ for some $U \in {\cal U}$, $\rho  \in \Sigma_m$ and
$\tau' \in \Sigma_n$

\noindent Denote by C the conclusion of the lemma, i.e. there is
$M'\prec M$ and $\sigma'$ such that
 $M'[\sigma'] \in SN$, and
 $(M'[\sigma'] \;\; P) \not\in
  SN$.

\noindent We prove something more general.

\noindent (1) Let $U \in {\cal U}$ and $\rho \in \Sigma_m$. Assume
$U[\rho]\in SN$ and $U[\rho][\delta=_rP]\not\in SN$. Then, C
holds.

\noindent (2) Let $V \in {\cal V}$ and $\tau \in \Sigma_n$. Assume
$V[\tau]\in SN$ and $V[\tau][\delta=_rP]\not\in SN$. Then, C
holds.

The conclusion C follows from (1) with $M$ and $\sig$. The properties
(1) and (2) are proved by a simultaneous induction on $\eta
c(U[\rho])$ (for the first case) and $\eta c(V[\tau])$ (for the second
case).

\medskip

Look first at (1)

\noindent - if $U =\l x U'$ or $U=\m\al U'$:  the result follows
 from the induction hypothesis with $U'$ and $\rho$.

\noindent - if $U= (U_1 \; U_2)$: if $U_i[\rho][\delta=_rP]\not\in
SN$ for $i=1$ or $i=2$, the result follows
 from the induction hypothesis with $U_i$ and $\rho$. Otherwise, by lemma \ref{l7} and
 \ref{l12},
 say $U_1 \ras \m\al U'_1$ and, letting $U'=U'_1[\al=_rU_2]$, $U'[\rho][\delta=_rP]\not\in SN$
 and the result follows
 from the induction hypothesis with $U'$ and $\rho$.

\noindent - if $U=(\delta \; U_1)$: if $U_1[\rho][\delta=_rP] \in
SN$, then $M'=U_1$ and $\sig'=\rho[\delta=_rP]$ satisfy the
desired conclusion. Otherwise, the result follows
 from the induction hypothesis with $U_1$ and $\rho$.

\noindent - if $U=(\al \; U_1)$: if $\al \not\in dom(\rho)$ or
$U_1[\rho][\delta=_rP] \not\in SN$, the result follows
 from the induction hypothesis with $U_1$ and $\rho$.  Otherwise,  let
 $\rho(\al)=V[\tau]$. If $V[\tau][\delta=_rP] \not\in SN$, the
 result follows
 from the induction hypothesis with $V$ and $\tau$ (with (2)). Otherwise, by lemmas \ref{l7} and
 \ref{l12}, there are two cases to consider.

 - $U_1 \ras \m\al_1 U_2$ and $U_2[\rho'][\delta=_rP] \not\in SN$
 where $\rho' =\rho + [\al_1=_rV[\tau]]$. The result follows from the
 induction hypothesis with $U_2$ and $\rho'$.

 - $V \ras \m\be V_1$ and $V_1[\tau'][\delta=_rP] \not\in SN$ where $\tau'=\tau+[\be=_lU_1[\rho]]$. The result follows
 from the induction hypothesis with $V_1$ and $\tau'$ (with (2)).

\medskip

The case (2) is proved in the same way. Note that, since $\delta$
is not free in the $N_i$, the case $b=(\delta \; V_1)$ does not
appear.
\end{proof}

\begin{theorem}\label{thm1}
Every term is in $SN$.
\end{theorem}

\begin{proof}
By induction on the term. It is enough to show that, if $M,N \in
SN$, then $(M \; N) \in SN$. We prove something more general: let
$\sigma$ (resp. $\tau$) be in $\Sigma_r$ (resp. $\Sigma_l$) and
assume $M[\sigma],N[\tau] \in SN$. Then $(M[\sigma] \; N[\tau])
\in SN$. Assume it is not the case and choose some elements such
that $M[\sigma],N[\tau] \in SN$, $(M[\sigma] \; N[\tau]) \not\in
SN$ and $(\eta(M)+\eta(N), cxty(M)+cxty(N))$ is minimal. By lemma
\ref{l12}, either $M[\sigma] \ras \mu \delta M_1$ and $M_1[\delta
=_r N[\tau]] \not \in SN$ or $N[\tau] \ras \mu \be N_1$ and
$N_1[\be =_l M[\sigma]] \not \in SN$. Look at the first case (the
other one is similar). By lemma \ref{l7}, $M \ras \mu \delta M_2$
for some $M_2$ such that $M_2[\sigma] \ras M_1$. Thus,
$M_2[\sigma][\delta =_r N[\tau]] \not \in SN$. By lemma \ref{crux}
with $M_2, \sig$ and $N[\tau]$, let $M'\prec M_2$ and $\sig'$ be
such that $M'[\sig'] \in SN$, $(M'[\sig'] \; N[\tau]) \not\in SN$.
 This contradicts the
minimality of the chosen elements since $\eta c(M') < \eta c(M)$.
\end{proof}

\subsection{The simply typed symmetric
$\la\m$-calculus is strongly normalizing}\label{beta}

In this section, we consider the simply typed calculus with the
$\be\m\m'$-reduction i.e. $M \tr M'$ means $M'$ is obtained from
$M$ by one step of the $\be\m\m'$-reduction. The strong
normalization of the $\be\m\m'$-reduction is proved essentially as
in theorem \ref{1thm1}.

 There is, however,  a new difficulty
:  a $\be$-substitution may create a $\m$, i.e. the fact that
$M[x:=N] \ras \mu \al P$ does not imply that $M \ras \mu \al Q$.
Moreover the $\m$ may come from a complicated interaction between
$M$ and $N$ and, in particular, the alternation between $M$ and
$N$ can be lost. Let e.g. $M=(M_1 \; (x \; (\la y_1 \la y_2 \m \al
M_4) \; M_2 \; M_3))$ and $N=\la z (z \; N_1)$. Then $M[x:=N]
\tr^* (M_1 \; (\m\al M'_4 \; M_3)) \ras \m\al M'_4[\al
=_rM_3][\al=_lM_1]$. To deal with this situation, we  need to
consider some new kind of $\m\m'$-substitutions (see definition
\ref{def}). Lemma \ref{l8b} gives the different ways in which a
$\m$ may appear. The difficult case in the proof (when a $\m$ is
created and the control between $M$ and $N$ is lost) will be
solved by using a typing argument.

To simplify the notations, we do not write explicitly the type
information but, when needed, we denote by $type(M)$ the type of
the term $M$.

\begin{lemma}\label{lammu}
\begin{enumerate}
\item If $(M \; N) \ras \l x P$, then $M \ras \l y M_1$ and
$M_1[y := N] \ras \l x P$.
\item If $(M \; N) \ras \m \al P$, then either ($M \ras \l y M_1$ and
$M_1[y := N] \ras \m \al P$) or ($M \ras \m \al M_1$ and $M_1[\al
=_r N] \ras P$) or ($N \ras \m \al N_1$ and $N_1[\al =_l M] \ras
P$).
\end{enumerate}
\end{lemma}
\begin{proof}
(1) is trivial. (2) is as in lemma \ref{ll2}.
\end{proof}

\begin{lemma}\label{l8a}
Let $M \in SN$ and $\sig=[x_1:=N_1, ..., x_k:=N_k]$. Assume
$M[\sig] \ras \la y P$. Then, either $M \ras \la y P_1$ and
$P_1[\sig] \ras P$
 or $M \ras  (x_i \; \overrightarrow{Q})$ and $(N_i \;
  \overrightarrow{Q[\sig]})\ras \la y P$.

\end{lemma}
\begin{proof}
By induction on $\eta c(M)$. The only non immediate case is $M=(R
\; S)$. By lemma \ref{lammu}, there is a term $R_1$ such that
$R[\sig] \ras \la z R_1$ and $R_1[z:=S[\sig]] \ras \la y P$. By
the induction hypothesis (since $\eta c(R) < \eta c(M)$), we have
two cases to consider.

(1) $R \ras \la z R_2$ and $R_2[\sig] \ras R_1$, then
$R_2[z:=S][\sig] \ras \la y P$. By the induction hypothesis (since
$\eta(R_2[z:=S])< \eta(M)$),

- either $R_2[z:=S] \ras \la y P_1$ and $P_1[\sig] \ras P$ ; but
then $M \ras \la y P_1$ and we are done.

- or $R_2[z:=S] \ras (x_i \; \overrightarrow{Q})$ and $(N_i \;
\overrightarrow{Q[\sig]})\ras \la y P$, then $M \ras (x_i \;
\overrightarrow{Q})$ and again we are done.

(2) $R \ras (x_i \; \overrightarrow{Q})$ and $(N_i \;
  \overrightarrow{Q[\sig]})\ras \la z R_1$. Then $M \ras (x_i \; \overrightarrow{Q} \;
  S)$ and the result is trivial.
\end{proof}

\begin{definition}\rm\label{def}
\begin{itemize}
  \item An  address is a finite list of symbols in
  $\{l,r\}$. The empty list is denoted by $[]$ and, if $a$ is an
  address and $s \in \{l,r\}$, $[s::a]$ denotes the list obtained
  by putting $s$ at the beginning of $a$.
  \item Let $a$ be an  address and $M$ be a term. The
  sub-term of $M$ at the address $a$ (denoted as $M_a$) is defined
  recursively as follows : if $M=(P \; Q)$ and $a=[r::b]$ (resp. $a=[l::b]$) then
  $M_a=Q_b$ (resp. $P_b$) and undefined otherwise.
  \item Let $M$ be a term and $a$ be an  address
  such that $M_a$ is defined. Then $M\langle a=N \rangle$ is the
  term $M$ where the sub-term $M_a$ has been replaced by
  $N$.
  \item Let $M,N$ be some terms and $a$ be an  address
  such that $M_a$ is defined. Then $N[\al=_aM]$ is the term $N$ in
  which each sub-term of the form $(\al \; U)$ is replaced by
  $(\al \; M\langle a=U\rangle)$.

\end{itemize}
\end{definition}

\begin{remark}

- Let $N = \la x (\al \; \la y (x \; \mu \b (\al \; y)))$, $M = (M_1 \; (M_2
\; M_3))$ and $a=[r::l]$. Then $N[\al=_aM] = \la x (\al \;
(M_1 \; (\la y (x \; \mu \b (\al \;(M_1 \; (y \; M_3)) )) \; M_3)))$.

- Let $M=(P \; ((R \; (x \; T)) \; Q))$ and $a=[r::l::r::l]$. Then
$N[\al=_aM]= N[\al=_r T][\al=_lR][\al=_r Q][\al=_rP]$.

- Note that the sub-terms of a term having an address in the sense
given above are those for which the path to the root consists only
on applications (taking either the left or right son).

 - Note
that $[\al=_{[l]}M]$ is not the same as $[\al=_lM]$ but
$[\al=_lM]$ is the same as $[\al=_{[r]}(M \; N)]$ where $N$ does
not matter. More generally, the term $N[\al=_aM]$ does not depend
of $M_a$.

- Note that  $M\langle a=N \rangle$ can be written as $M'[x_a:=N]$
where $M'$ is the term $M$ in which $M_a$ has been replaced by the
fresh variable $x_a$ and thus (this will be used in the proof of
lemma \ref{thm}) if $M_a$ is a variable $x$, $(\al \; U)[\al=_aM]=
(\al \; M_1[y:=U[\al=_aM]])$ where $M_1$ is the term $M$ in which
the particular occurrence of $x$ at the address $a$ has been
replaced by the fresh name $y$ and the other occurrences of $x$
  remain unchanged.

  \end{remark}

 \begin{lemma}\label{new}
Let $M$ be a term and $\sigma = [\al_1=_{a_1}N_1, ...,
  \al_n=_{a_n}N_n]$.
\begin{enumerate}
\item If $M[\sigma]  \ras \l x P$, then $M \ras \l x Q$ and
$Q[\sigma]  \ras P$.
\item If $M[\sigma]   \ras \m \al P$, then $M \ras \m \al Q$ and
$Q[\sigma]  \ras P$.
\end{enumerate}
\end{lemma}
\begin{proof}
By induction on $M$. Use lemma \ref{lammu}.
\end{proof}

  \begin{lemma}\label{l30}
Assume $M,N \in SN$ and $(M \; N) \not\in SN$. Then, either ($M
\ras \la y P$ and $P[y:=N] \not\in SN$) or
  ($M \ras \mu \al P$ and $P[\al =_r N] \not\in SN$) or
 ($N \ras \mu \al P$ and $P[\al =_l M] \not\in SN$).
\end{lemma}
\begin{proof}
By induction on $\eta(M)+\eta(N)$.
\end{proof}

\begin{lemma}
If $\G \v M : A$ and $M \ras N$ then $\G \v N : A$.
\end{lemma}

\begin{proof}
Straightforward.
\end{proof}

\begin{lemma}\label{l8b}
Let $n$ be an integer,  $M \in SN$, $\sig=[x_1:=N_1, ...,
x_k:=N_k]$ where $lg(type(N_i))=n$ for each $i$.  Assume $M[\sig]
\ras \mu \al P$. Then,
\begin{enumerate}
  \item either $M \ras \mu \al P_1$ and $P_1[\sig] \ras P$
  \item  or $M \ras  Q$ and, for some $i$, $N_i \ras \m\al N'_i$  and $N'_i[\al=_aQ[\sig]] \ras
P$ for some
 address $a$ in $Q$ such that  $Q_a=x_i$.
  \item or $M \; \ras  Q$,  $Q_a[\sig] \ras \m \al N'$ and
$N'[\al=_aQ[\sig]] \ras P$ for some
 address $a$ in $Q$ such that $lg(type(Q_a)) < n$ .
\end{enumerate}
\end{lemma}
\begin{proof}
By induction on $\eta c(M)$.   The only non immediate case is
$M=(R \; S)$. Since $M[\sig] \ras \mu \al P$, the application
$(R[\sig] \; S[\sig])$ must be reduced. Thus there are three cases
to consider.

\begin{itemize}
  \item It is reduced by a $\m'$-reduction, i.e. there is a term $S_1$
  such that $S[\sig] \ras \mu \al S_1$ and $S_1[\al=_lR[\sig]] \ras
  P$.  By the induction hypothesis: \\ - either $S \ras \m\al Q$ and
  $Q[\sig] \ras S_1$, then $M \ras \m \al Q[\al=_lR]$ and
  $Q[\al=_lR][\sig] \ras P$.\\ - or $S \ras Q$ and, for some $i$, $N_i \ras \m\al N'_i$,
  $Q_a=x_i$ for some address $a$ in $Q$ and $N'_i[\al=_aQ[\sig]] \ras S_1$.
  Then $M \ras (R \; Q)=Q'$ and letting $b=[r::a]$ we have
  $N'_i[\al=_bQ'[\sig]] \ras P$. \\ - or $S \;\ras Q$, $Q_a[\sig] \ras
  \m \al N'$ for some address $a$ in $Q$ such that $lg(type(Q_a)) <
  n$ and $N'[\al=_aQ[\sig]] \ras S_1$. Then $M \ras (R \; Q)=Q'$ and
  letting $b=[r::a]$ we have $N'[\al=_bQ'[\sig]] \ras P$ and
  $lg(type(Q'_b))<n$.
  \item It is reduced by a $\m$-reduction. This
  case is similar to the previous one.
   \item It is reduced by a
  $\be$-reduction, i.e. there is a term $U$ such that $R[\sig] \ras
  \la y U$ and $U[y:=S[\sig]] \ras \mu \al P$. By lemma \ref{l8a},
  there are two cases to consider. \\ - either $R \ras \la y R_1$ and
  $R_1[\sig][y:=S[\sig]]=R_1[y:=S][\sig] \ras \mu \al P$. The result
  follows from the induction hypothesis since $\eta(R_1[y:=S]) <
  \eta(M)$.\\ - or $R \ras (x_i \; \overrightarrow{R_1})$. Then $Q=(x_i \;
\overrightarrow{R_1} \; S)$ and $a=[]$ satisfy the desired
conclusion since then $lg(type(M)) <n$.\qed
\end{itemize}
\end{proof}

\begin{definition}\rm
Let $A$ be a type. We denote by $\Sigma_A$ the set of
substitutions of the form $[\al_1=_{a_1}M_1, ...,
  \al_n=_{a_n}M_n]$  where the type of the $\al_i$ is $\neg A$.
\end{definition}

\begin{remark}

Remember that the type of $\al$ is not the same in $N$ and in
$N[\al=_{a}M]$. The previous definition may thus be considered as
ambiguous. When  we consider the term $N[\sig]$ where $\sig \in
\Sigma_A$, we assume that $N$ (and not $N[\sig]$) is typed in the
context where the $\al_i $ have type $ A$. Also note that
considering $N[\al=_aM]$ implies that the type of $M_a$ is $A$.
\end{remark}

\begin{lemma}\label{crux'}
Let $n$ be an integer and $A$ be a type such that $lg(A)=n$. Let
$N,P$ be terms and $\tau \in \Sigma_A$. Assume that,
\begin{itemize}
  \item  for every $M,N \in SN$ such that $lg(type(N))<n$,
$M[x:=N] \in SN$.
  \item $N[\tau] \in SN$ but $N[\tau][\delta =_aP] \not \in SN$.
  \item $\delta$ is free and has type $\neg A$ in $N$ but $\delta$ is not free in
$Im(\tau)$.
\end{itemize}

\noindent Then,  there is $N'\prec N$ and $\tau' \in \Sigma_A$
such that
 $N'[\tau'] \in SN$ and
 $P\langle a = N'[\tau']\rangle\not\in
  SN$.

\end{lemma}
\begin{proof}
The proof looks like the one of lemma \ref{crux}. Denote by (H)
the first assumption i.e. for every $M,N \in SN$ such that
$lg(type(N))<n$, $M[x:=N] \in SN$.

Let $\tau =[\al_1=_{a_1}M_1, ...,
  \al_n=_{a_n}M_n]$, ${\cal U}=\{U \; / \; U \peq N\}$ and
${\cal V}=\{V \; / \; V \peq M_i$ for some $i\}$. Define
inductively the sets $\Sigma_m$ and $\Sigma_n$ of substitutions by
the following rules:

$\rho \in \Sigma_n$ iff $ \rho = \es$ or $\rho=\rho'  + [\al =_a
V[\sig]]$ for some $V \in {\cal V} $, $\sig  \in \Sigma_m$, $\rho'
\in \Sigma_n$ and $\al$ has type $\neg A$.

$\sig \in \Sigma_m$ iff $ \sig = \es$ or $\sig=\sig' + [x:=
U[\rho]]$ for some $U \in {\cal U}$, $\rho  \in \Sigma_n$, $\sig'
\in \Sigma_m$ and $x$ has type $A$.

\noindent Denote by C the conclusion of the lemma. We prove
something more general.

\noindent (1) Let $U \in {\cal U}$ and $\rho \in \Sigma_n$. Assume
$U[\rho]\in SN$ and $U[\rho][\delta=_aP]\not\in SN$. Then, C
holds.

\noindent (2) Let $V \in {\cal V}$ and $\sig \in \Sigma_m$. Assume
$V[\sig]\in SN$ and $V[\sig][\delta=_aP]\not\in SN$. Then, C
holds.

Note that the definitions of the sets $\Sigma_n$ and $\Sigma_m$
are not the same as the ones of lemma \ref{crux}. We gather here
in $\Sigma_n$ all the $\mu\mu'$-substitutions getting thus the new
substitutions of definition \ref{def} and we put in $\Sigma_m$
only the $\l$-substitutions.

The conclusion C follows from (1) with $N$ and $\tau$. The properties
(1) and (2) are proved by a simultaneous induction on $\eta
c(U[\rho])$ (for the first case) and $\eta c(V[\tau])$ (for the second
case).

 The proof is by case analysis as in
lemma \ref{crux}. We only consider the new case for $V[\sig]$,
i.e.  when $V=(V_1 \; V_2)$ and $V_i[\sig][\delta=_aP]\in SN$. The
other ones are done essentially in the same way as in lemma
\ref{crux}.

\medskip

\noindent - Assume first the interaction between $V_1$ and $V_2$
is a $\be$-reduction. If $V_1 \ras \l x V'_1$, the result follows
from the induction hypothesis with $V'_1[x:=V_2][\sig]$.
Otherwise, by lemma \ref{l8a}, $V_1 \ras (x \;
\overrightarrow{W})$. Let $\sig(x)=U[\rho]$. Then  $(U[\rho]
\;\overrightarrow{W}[\sig]) \ras \l y Q$ and
$Q[y:=V_2[\sig]][\delta=_aP]\not\in SN$. But, since the type of $x$ is
$A$, the type of $y$ is less than $A$ and since $Q[\delta=_aP]$ and
$V_2[\sig][\delta=_aP]$ are in $SN$ this contradicts (H).

\noindent - Assume next the interaction between $V_1$ and $V_2$ is
a $\m$ or $\m'$-reduction.  We consider only the case $\m$ (the
other one is similar). If $V_1 \ras \m \al V'_1$, the result
follows from the induction hypothesis with
$V'_1[\al=_rV_2][\sig]$. Otherwise, by lemma \ref{l8b}, there are
two cases to consider.

- $V_1 \ras Q$, $Q_c=x$ for some address $c$ in $Q$ and $x \in
dom(\sig)$, $\sig(x)= U[\rho]$, $U[\rho] \ras \m\al U_1$ and
$U_1[\al=_cQ[\sig]] [\al=_r V_2[\sig]][\delta=_aP]\not\in
SN$. By lemma \ref{new}, we have $U \ras \m\al U_2$ and $U_2[\rho]
\ras U_1$, then $U_2[\rho][\al=_cQ[\sig]] [\al=_r V_2[\sig]][\delta=_aP]\not\in
SN$. Let $V'=(Q \; V_2)$ and $b=l::c$. The result follows then from
the induction hypothesis with $U_2[\rho']$ where $\rho'=\rho +
[\al=_bV'[\sig]]$.

- $V_1 \; \ras Q$, $Q_c[\sig][\delta=_aP] \ras \m \al R$ for some
address $c$ in $Q$ such that $lg(type(Q_c)) < n$,
$R[\al=_cQ[\sig][\delta=_aP]] [\al=_r
V_2[\sig][\delta=_aP]]\not\in SN$. Let $V'=(Q' \; V_2)$ where $Q'$
is the same as $Q$ but $Q_c$ has been replaced by a fresh variable
$y$ and $b=l::c$. Then $R[\al=_bV'[\sig][\delta=_aP]] \not \in
SN$. Let $R'$ be such that $R' \prec R$, $R'
[\al=_bV'[\sig][\delta=_aP]] \not \in SN$ and $\eta c(R')$ is
minimal. It is easy to check that $R'=(\al \; R'')$,
$R''[\al=_bV'[\sig][\delta=_aP]] \in SN$ and
$V'[\sig'][\delta=_aP] \not\in SN$ where $\sig' =\sig
+ [y:=R''[\al=_bV'[\sig]]]$. If $V'[\sig][\delta=_aP] \not\in SN$, we
get the result by the induction hypothesis since $\eta c(V'[\sig])
< \eta c (V[\sig])$. Otherwise this contradicts the assumption (H)
since $V'[\sig][\delta=_aP], R''[\al=_bV'[\sig][\delta=_aP]] \in
SN$, $V'[\sig][\delta=_aP][y:= R''[\al=_bV'[\sig][\delta=_aP]]]
\not\in SN$ and the type of $y$ is less than $n$.
\end{proof}

\begin{lemma}\label{thm}
If $M,N \in SN$, then $M[x:=N] \in SN$.
\end{lemma}
\begin{proof}
We prove something a bit more general: let $A$ be a type, $M, N_1,
..., N_k$ be terms and $\tau_1, ..., \tau_k $ be substitutions in
$\Sigma_A$. Assume that, for each $i$,  $N_i$ has type $A$ and
$N_i[\tau_i] \in SN$. Then $M[x_1:=N_1[\tau_1], \ptv
x_k:=N_k[\tau_k]] \in SN$. This is proved by induction on $(lg(A),
\eta(M), cxty(M)$, $\Sigma \; \eta(N_i), \Sigma \; cxty(N_i))$
where, in $\Sigma \; \eta(N_i)$ and $\Sigma \; cxty(N_i)$, we
count each occurrence of the substituted variable. For example if
$k=1$ and $x_1$ has $n$ occurrences, $\Sigma \; \eta(N_i)=n.
\eta(N_1)$.

 If $M$ is $\lambda y M_1$ or $(\alpha \; M_1)$ or $\mu\alpha M_1$
or a variable, the result is trivial.
  Assume then that $M=(M_1 \;
M_2)$. Let $\sigma = [x_1:=N_1[\tau_1], \ptv x_k:=N_k[\tau_k]]$.
By the induction hypothesis, $M_1[\sigma], M_2[\sigma] \in SN$. By
lemma \ref{l30} there are 3 cases to consider.

\begin{itemize}
  \item  $M_1[\sigma] \ras \lambda y P$ and $P[y:=M_2[\sigma]] \not\in
SN$. By lemma \ref{l8a}, there are two cases to consider.

\begin{itemize}
  \item $M_1 \ras \lambda y Q$ and $Q[\sigma] \ras P$. Then
  $Q[y:=M_2][\sigma]= Q[\sigma][y:=M_2[\sigma]] \ras
  P[y:=M_2[\sigma]]$ and, since $\eta(Q[y:=M_2]) < \eta(M)$, this contradicts the induction hypothesis.
  \item $M_1 \ras (x_i \; \overrightarrow{Q})$ and $(N_i \;
  \overrightarrow{Q[\sigma]}) \ras \lambda y P$. Then, since the type of $N_i$ is $A$,
 $lg(type(y))<lg(A)$. But $P, M_2[\sigma]
  \in SN$ and $P[y:=M_2[\sigma]] \not\in
SN$. This contradicts the induction hypothesis.

\end{itemize}

  \item $M_1[\sigma] \ras \mu\alpha P$ and $P[\alpha=_rM_2[\sigma]]
  \not\in SN$. By lemma \ref{l8b}, there are three cases to consider.

\begin{itemize}
  \item $M_1 \ras \mu\alpha Q$ and $Q[\sigma] \ras P$. Then, $Q[\alpha=_rM_2][\sigma]= Q[\sigma][\alpha=_rM_2[\sigma]] \ras
  P[\alpha=_rM_2[\sigma]]$ and, since $\eta(Q[\alpha=_rM_2]) < \eta(M)$, this contradicts the induction hypothesis.
  \item $M_1 \ras  Q$, $N_i [\tau_i] \ras \m\alpha L'$ and  $Q_a=x_i$ for some
 address $a$ in $Q$ such that $L'[\alpha=_aQ[\sigma]] \ras
P$ and thus $L'[\al=_bM'[\sigma]]  \not\in SN$ where $b=(l::a)$ and
$M'=(Q \; M_2)$.

By lemma \ref{l7}, $N_i \ras \mu\al L$ and $L[\tau_i] \ras L'$.
Thus, $L[\tau_i][\al=_bM'[\sigma]]  \not\in SN$. By lemma
\ref{crux'}, there is $L_1 \prec L$ and $\tau'$ such that
$L_1[\tau'] \in SN$ and $M'[\sigma]\langle
b=L_1[\tau']\rangle\not\in SN$. Let $M''$ be $M'$ where the
variable $x_i$ at the address $b$ has been replaced by the fresh
variable $y$ and let $\sig_1=\sig + [y:=L_1[\tau']]$. Then
$M''[\sig_1]=M'[\sigma]\langle b=L_1[\tau']\rangle\not\in SN$.

If $M_1 \ra^+  Q$ we get a contradiction from the induction
hypothesis since $\eta(M'') < \eta(M)$. Otherwise, $M''$ is the
same as $M$ up to the change of name of a variable and $\sigma_1$
differs from $\sigma$ only at the address $b$. At this address,
$x_i$ was substituted in $\sigma$ by $N_i[\tau_i]$ and in
$\sigma_1$ by $L_1[\tau']$ but $\eta c(L_1) < \eta c(N_i)$ and
thus we get a contradiction from the induction hypothesis.

  \item $M \ras  Q$,  $Q_a[\sigma] \ras \m \al L$ for some
 address $a$ in $Q$ such that $lg(type(Q_a)) < lg(A)$  and
$L[\al=_aQ[\sigma]] \ras P$. Then, $L[\al=_bM'[\sigma]] \not \in
SN$ where $b=[l::a]$ and $M'=(Q \; M_2)$.

By lemma \ref{crux'}, there is an $ L'$ and  $\tau'$ such that
$L'[\tau'] \in SN$ and $M'[\sigma]\langle b=L'[\tau'] \rangle\not
\in SN$. Let $M''$ be $M'$ where the variable $x_i$ at the address
$b$ has been replaced by the fresh variable $y$. Then $M''[\sig][
y:=L'[\tau']]=M'[\sigma]\langle b=L'[\tau']\rangle\not\in SN$.

But $\eta(M'') \leq \eta(M)$ and $cxty(M'') < cxty(M)$ since,
because of its type, $Q_a$ cannot be a variable and thus, by the
induction hypothesis, $M''[\sigma]\in SN$. Since $M''[\sigma ][
y:=L'[\tau']]\not \in SN$ and $lg(type(L')) < lg(A)$, this
contradicts the induction hypothesis.

\end{itemize}

  \item $M_2[\sigma] \ras \mu\alpha P$ and $P[\alpha=_lM_1[\sigma]]
  \not\in SN$. This case is similar to the previous one.\qed
\end{itemize}
\end{proof}

\begin{theorem}\label{th}
Every typed term is in $SN$.
\end{theorem}
\begin{proof}
 By induction on the term. It is enough to show that if $M,N \in SN$, then $(M \; N)
\in SN$. Since $(M \; N)=(x \; y)[x:=M][y:=N]$ where $x,y$ are
fresh variables, the result follows by applying  theorem \ref{thm}
twice and the induction hypothesis.
\end{proof}

\section{Remarks and future work}

\subsection{Why the usual candidates do not work ?}\label{s3}

In \cite{Par2}, the proof of the strong normalization of the
$\la\m$-calculus is done by using the {\em usual} (i.e. defined
without a fix-point operation) candidates of reducibility. This
proof could be easily extended to the symmetric $\la\m$-calculus
if we knew the following properties for the un-typed calculus:

\begin{enumerate}
  \item If $N$ and $(M[x:=N] \; \overrightarrow{P})$ are in $SN$,
  then so is $(\la x M \; N \; \overrightarrow{P})$.
  \item If $N$ and $(M[\al=_rN] \; \overrightarrow{P})$ are in $SN$,
  then so is $(\m \al M \; N \; \overrightarrow{P})$.
  \item If $\overrightarrow{P}$ are in $SN$, then so is $(x \; \overrightarrow{P})$.
\end{enumerate}

These properties are easy to show for the $\be\m$-reduction but
they were not known for the $\be\m\m'$-reduction.

The third property is true but the properties (1) and (2) are
false. The proof of (3) and the counter-examples for (1) and (2)
can be found in \cite{craco}.

\subsection{Future work}\label{s7}

 We believe that our technique, will allow to give explicit bounds for
  the length of the reductions of a typed term. This is a goal we
  will try to manage.

\end{document}